\documentclass[12pt]{article}
\usepackage{color,xcolor}
\usepackage{graphicx}
\usepackage{amssymb}
\usepackage{amsthm}
\usepackage{amsmath,numbysec}
\usepackage[margin=1in]{geometry}
\usepackage{caption}
\usepackage{subcaption}
\usepackage{float}
\usepackage{mdframed,pdfsync, mathrsfs}

\newtheorem{thm}{Theorem}[section]
\newtheorem{lemma}[theorem]{Lemma}
\newtheorem{proposition}[theorem]{Proposition}

\newcommand{\bea}{\begin{eqnarray*}}
\newcommand{\eea}{\end{eqnarray*}}
\newcommand{\ben}{\begin{eqnarray}}
\newcommand{\een}{\end{eqnarray}}
\newcommand{\beq}{\begin{equation}}
\newcommand{\eeq}{\end{equation}}

\newcommand{\osc}{\text{osc}}

\newcommand{\R}{\ensuremath{\mathbb{R}}}
\newcommand{\N}{\ensuremath{\mathbb{N}}}

\newcommand{\T}{\ensuremath{\mathbb{T}}}
\newcommand{\Rm}{{\mathbb R}}

\newcommand{\farc}{\frac}

\renewcommand{\d}{\partial}

\def\red{\textcolor{red}}

\begin{document}
\title{Global well-posedness for the Euler alignment system with mildly singular interactions}
\author{Jing An\footnote{Institute for Computational and Mathematical Engineering, Stanford University, Stanford, CA 94305, USA;
jingan@stanford.edu}\and Lenya Ryzhik\footnote{Department of Mathematics, Stanford University, Stanford, CA 94305, USA; ryzhik@stanford.edu}}
\maketitle
\numberbysection
\def\red{\textcolor{red}}
\begin{abstract}
We consider the Euler alignment system with mildly singular interaction kernels. When the 
local repulsion term is of the fractional type, global in time
existence of smooth solutions was proved 
in~\cite{do2018global,shvydkoy2017eulerian1,shvydkoy2017eulerian2,shvydkoy2017eulerian3}. 
Here, we consider a class of less singular interaction kernels
and establish the global regularity of solutions as long as the interaction kernels are not integrable. 
The proof relies on modulus of continuity estimates for a class of parabolic 
integro-differential equations with a drift and mildly singular kernels. 
\end{abstract}

\section{Introduction}

\subsubsection*{The Euler alignment system}

The Cucker-Smale model \cite{cucker-smale07}  
\begin{equation}\label{jan602}
\dot{x}_i=v_i,\quad \dot{v}_i=\frac{1}{N}\sum_{j=1}^N\psi(|x_i-x_j|)(v_j-v_i)
\end{equation}
describes the dynamics of a flock of $N$ individuals (birds, fish, etc.) that tend to align their velocities locally.
Here, $x_i$ and $v_i$ are  the position and the velocity of the $i$-th individual in the flock. 
The non-negative ``influence function"~$\psi(r)\ge 0$, measures the strength of the alignment and is a
decreasing function of $r$. By now, it is one of the standard models for the flocking
phenomenon -- emergence of self-organized groups (flocks) that move as a group -- 
see~\cite{carrillochoiperez,degond-dimarco-mac-wang,Vicsek2012} 
for a review. 

The Euler alignment system    
\begin{align}
\label{euler1} 
&\partial_t \rho+\partial_x(\rho u) =0\\
\label{euler2} 
&\partial_t u+ u\partial_x u = \int_{\mathbb{R}}\psi(|x-y|)[u(t,y)-u(t,x)] \rho(t,y)dy.
\end{align}
is a hydrodynamic limit of the Cucker-Smale system (\ref{jan602}), in the regime where the number~$N$ of the individuals is very large,
and they are already locally aligned,
so that their evolution may be described   in terms of a local density $\rho(t,x)$ and a local velocity $u(t,x)$. In the absence
of the local alignment, when $\psi\equiv 0$, the velocity equation (\ref{euler2}) is simply the inviscid Burgers' equation that may develop a discontinuity in
$u(t,x)$ in a finite time. In terms of the flocking dynamics, this corresponds to a collision of two flocks that can easily happen when
there is no local tendency to align.  
On the intuitive level, a positive interaction kernel~$\psi>0$ promotes a local alignment and fights the shock creation. 
However, it was shown in~\cite{carrillo2016critical,tadmor2014critical}
that if the kernel $\psi(r)$ is Lipschitz, then the solutions of the Euler alignment system may still develop a discontinuity in~$u(t,x)$ in a finite
time, though the class of the initial conditions that lead to a discontinuity is smaller than for the Burgers' equation, so even
a Lipschitz interaction kernel $\psi(r)$ has some regularizing effect. More precisely, 
if $\psi(r)$ is Lipschitz, then solutions of the
Euler alignment system remain regular if and only if the initial conditions~$u_0(x)=u(0,x)$ and $\rho_0(x)=\rho(0,x)\ge 0$ satisfy
\begin{equation}\label{dec3002}
\d_x u_0(x) \ge -(\psi\star\rho_0)(x)\hbox{ for all $x\in \R$}.
\end{equation}
Otherwise, $u(t,x)$ develops a discontinuity in a finite time. This is a natural generalization of the classical criterion for regularity of the
solutions of the  Burgers' equation with $\psi=0$.   Note that $\rho_0\ge 0$
from physical considerations.

Recently there has been an increased interest in alignment kernels $\psi(r)$ that are singular at $r\downarrow 0$, so that the local
alignment effect is  much stronger than for the Lipschitz kernels, both for the Cucker-Smale 
and the Euler alignment systems -- 
see~\cite{carrillo-choi-mucha-peszek2016,do2018global,mucha-peszek2015,peszek2014,peszek2015,shvydkoy2017eulerian1,shvydkoy2017eulerian2,shvydkoy2017eulerian3} and 
references therein. In  light of the regularity condition (\ref{dec3002}) for the Lipschitz interaction kernels, it is natural to conjecture   
that solutions of the Euler alignment system (\ref{euler1})-(\ref{euler2}) remain smooth for all times $t>0$, provided that the interaction
kernel $\psi\ge 0$ is not integrable, as then (\ref{dec3002}) holds automatically, as long as $\rho_0\not\equiv 0$.
In this direction, the global existence of smooth solutions of (\ref{euler1})-(\ref{euler2}) for singular interaction kernels  of the 
form $\psi(x) = C|x|^{-1-\alpha}$, with $\alpha\in (0,2)$ was proved 
in~\cite{do2018global,shvydkoy2017eulerian1,shvydkoy2017eulerian2,shvydkoy2017eulerian3}. We note that the particular 
scaling properties of such kernels are important for the regularity proofs, especially in~\cite{do2018global}.  We also mention that
a qualitatively similar 
nonlinearly enhanced regularizing effect happens also in nonlinear porous medium problems and Keller-Segel  equations
~\cite{bedrossian2015,bedrossian-kim2013,bedrsossian-rodriguez2014,bedrossian-rodr-berotzzi2011,caffarelli2013regularity,caffarelli2011nonlinear}. 

%

\subsubsection*{The main results}

In this paper, we consider the Euler alignment system (\ref{euler1})-(\ref{euler2}) with 
$2\pi$-periodic initial conditions $\rho(0,x)=\rho_0(x)$, $u(0,x)=u_0(x)$, such that~$\rho_0(x)\ge c_0>0$ for all
$x\in\R$, and establish the global regularity of the solutions for a general class of interaction kernels $\psi(r)$ that are not integrable but 
blow-up much slower than $r^{-1-\alpha}$ as $r\downarrow 0$. We make the following assumptions on the interaction kernel $\psi$:
 
(i) For any $\alpha>0$, $\psi(r)$ is less singular than $1/r^{1+\alpha}$ but more singular than~$1/r^{1-\alpha}$,  
so that there exists $c_\alpha>0$ such that 
\begin{equation}\label{aug1002}
\farc{1}{c_\alpha r^{1-\alpha}}\le\psi(r)\le \farc{c_\alpha}{r^{1+\alpha}}\hbox{ for all $0<r\le 1$},
\end{equation}
and $\psi(r)$ is not integrable:
\begin{equation}\label{aug108}
M (r):= \int_{r}^{\infty}\psi(y)dy\to+\infty\hbox{ as $r\to 0$.}
\end{equation}
It follows from (\ref{aug1002}) that $M(r)$ is less singular 
than $1/r^{\alpha}$ for any $\alpha>0$:
\begin{equation}\label{aug1004}
\lim_{r\to 0} r^{\alpha} M(r) = 0.
\end{equation}
(ii) The function $\psi(r)$ is symmetric, decreasing and  
satisfies the H\"{o}rmander-Mikhlin type condition: there exists~$C>0$ so that 
\begin{equation} \label{aug106}
|r\psi'(r)|\le C \psi(r),
\end{equation}
and also that a doubling condition holds:
\begin{equation}\label{doubling}
\psi(r)\le C \psi(2r)\hbox{ for all $r>0$.}
\end{equation}
(iii) We also assume that 
there exists $r_0\le 1$ such that
\begin{equation}\label{aug1006}
\hbox{the ratio $\dfrac{r\psi(r)}{M(r)}$ is non-decreasing for $0<r<r_0$},
\end{equation}
and that there exists $\gamma\in(0,1/2]$ and $r_0>0$ such that 
\begin{equation}\label{aug1008}
\hbox{$r^{\gamma} M(r)$ is non-decreasing for $0\leq r\leq r_0$}.
\end{equation}
The last assumption is almost automatic since both 
$r^\gamma M(r)>0$ for all $r>0$ and~(\ref{aug1004}) holds. It follows that $rM(r)$ is 
also non-decreasing.  Note that we do not need to assume that $m(r)=r\psi(r)$ is singular at $r=0$ 
as in \cite{do2018global} and \cite{shvydkoy2017eulerian2}. The H\"ormander-Mikhlin condition is used in the proof of
Lemma~\ref{lem:aug14}, a version of the Constantin-Vicol nonlocal maximum principle, that allows us to control the density $\rho(t,x)$
in the $L^\infty$-norm, ensuring that the dissipative term is, indeed, dissipating. 
One may reasonably say that our assumptions cover most  ``well-behaved" 
non-integrable influence
functions $\psi(r)$.  
The main result of this paper is the following theorem. 
\begin{thm}\label{thm-main}
Under the above assumptions, the Euler alignment system (\ref{euler1})-(\ref{euler2}) with periodic smooth initial conditions $(\rho_0,u_0)$
such that $\rho_0(x)\ge c_0>0$, 
has a unique global in time smooth
solution $\rho(t,x)$, $u(t,x)$. 
\end{thm}
The strict positivity of the density is needed for ``unconditional" regularity: if there are regions such that $\rho_0(x)=0$
then a Burgers'-like mechanism may lead to blow up even for fractional-type influence kernels~\cite{tan2017}.
Let us also mention that when the influence kernel is integrable, the finite time blow-up scenario for Lipschitz influence kernels
in \cite{carrillo2016critical} still applies, even without the 
assumption that the influence function $\psi$ is Lipschitz. Indeed, since $\psi$ only shows up in the convolutions, and the quantities 
$\rho$ and $\d_x \rho$ that $\psi$ convolves with, by proof of contradiction, are assumed to stay bounded, the proof applies
for $\psi \in L^1(\Rm)$ as well. In that sense, Theorem~\ref{thm-main} is reasonably sharp, except for our assumptions above that $\psi$
is not just non-integrable but also ``nicely-behaved".  

In order to explain the proof of Theorem~\ref{thm-main}, we recall that 
the Euler alignment system~(\ref{euler1})-(\ref{euler2}) can be reformulated as
\begin{equation}\label{main2}
    \d_t \rho + u\d_x \rho+ \rho \mathcal{L}\rho = -G\rho,
\end{equation}
with $G:= \d_x u - \mathcal{L}\rho$, and the operator ${\cal L}$ given by
\begin{equation}\label{aug102}
\mathcal{L}f(x):=\int_{\mathbb{R}}\psi(x-y)(f(x)-f(y))dy.
\end{equation} 
As in~\cite{do2018global,shvydkoy2017eulerian1,shvydkoy2017eulerian2,shvydkoy2017eulerian3}, 
one may show that both the density $\rho$ and the function $G$ are uniformly bounded. Thus, (\ref{main2}) may be thought of 
as an integro-differential equation for $\rho(t,x)$ of the form
\begin{equation}\label{fullholder}
\d_t q + v(t,x)\cdot\nabla q + \mathscr{L} q = f(t,x),
\end{equation}
with a bounded function $f(t,x)$ and an operator $\mathscr L$ of the form
\begin{equation}\label{jan202}
\mathscr{L} q(t,x) = \int_{\mathbb{R}^d} (q(t,x)-q(t,x+z)) k(x,z,t) dz,
\end{equation}
with a kernel $k(x,z,t)$ that obeys bounds similar to $\psi$, when considered as a function of~$z$.
When~$k(x,z,t) = C|z|^{-d-\alpha}, \alpha\in(0,2)$, the operator $\mathscr{L}$ is the standard fractional Laplacian, 
and H\"{o}lder estimates for such time dependent fractional diffusion equations with a drift have been obtained 
in \cite{silvestre2012holder} using purely analytic techniques. For more general kernels, closer to our assumptions, elliptic estimates
in the absence of a drift are provided in~\cite{kassmann2017intrinsic}
applying a combination of anlytical and probabilistic methods. These estimates were extended in \cite{chen2014holder} to 
time dependent equations with a drift, using a purely probabilistic approach. Both~\cite{chen2014holder} and \cite{kassmann2017intrinsic} assume that $m(x) = |x|\psi(x)$ varies regularly at zero with index $\alpha \in \mathbb{R}$, in the sense that for every $\lambda>0$,
\begin{equation}\label{aug1040}
\lim_{r\to 0} \frac{m( r)}{m(\lambda r)} = \lambda^{\alpha},
\end{equation}
and rely on several properties derived from this assumption.  Here, we present an alternative analytical approach to 
the H\"{o}lder estimates for the parabolic equations with a drift, and a weaker than fractional dissipation, 
based on combining the methods in~\cite{kassmann2017intrinsic} with a version of the quantitative comparison principle 
in~\cite{silvestre2012holder}. This allows us to relax (\ref{aug1040}) to 
assumptions~(\ref{aug1002}), (\ref{doubling}) and (\ref{aug1006})-(\ref{aug1008}),
and obtain the following H\"older regularity estimate for the solutions to the Euler alignment system, that leads to Theorem~\ref{thm-main}. 
\begin{thm}\label{thm-holder}
Suppose the above assumptions (\ref{aug1002}), (\ref{aug108}) and (\ref{doubling})-(\ref{aug1008}) hold, and let $\rho(t,x)$ be 
a solution to the Euler alignment system (\ref{euler1})-(\ref{euler2}).
There exists $\beta \in (0,1)$, $r_0>0$ and a sufficiently small constant $c'>0$ such that for any $0<t\leq 1$ 
and $|x-y| \leq \min(r_0,M^{-1}(c'/t))$, we have
\begin{align}\label{thm-holder-result}
|\rho(t,x) - \rho(t,y)|\le K_0 t^{-{\beta}} [M(|x-y|)]^{-\beta},
\end{align}
with a constant $K_0$ that depends only on the initial conditions $\rho_0$ and $u_0$. 
\end{thm}
Let us note that, compared to \cite{silvestre2012holder}, we need to work with the advection $u(t,x)$ that is not Lipschitz but only $M$-Lipschitz
in space: there exists $C>0$ such that 
\[
|u(t,x)-u(t,y)|\le C|x-y|M(|x-y|),\hbox{ for $|x-y|\le 1$. }
\]
This is similar to log-Lipschitz velocities in the Yudovich theory for the Euler equation. 
%
   
A word on notation: we note by $C>0$ universal constants that may change from line to line. For important constants, we denote as $c',C', C_1,$ etc. to distinguish them. For higher order derivatives in $x$, we use $(n)$ to denote, for example, $\rho^{(n)} (t,x) = \d_x^n \rho(t,x)$.
The torus we use here is $\T=[-\pi,\pi]$. 

{\bf Acknowledgement.} JA was supported by the Oliger Memorial Graduate Fellowship, and LR by an NSF grant DMS-1613603. 

\section{Preliminaries}

In this section, we establish some preliminary results for the proof of Theorem~\ref{thm-main}. 

\subsection{A reformulation of the Euler alignment system}

Let us first recall a convenient reformulation of the Euler alignment system.
Applying the operator $\mathcal{L}$ to \eqref{euler1} gives:
\begin{align}
\label{reform1}0&= \partial_t(\mathcal{L}\rho)+\partial_x(\mathcal{L}(\rho u)) =\partial_t(\mathcal{L}\rho)
+\partial_x\int_{\mathbb{R}}\psi(x-y)[\rho(x)u(x)-\rho(y)u(y)] dy\\
\nonumber &=\partial_t(\mathcal{L}\rho)+\partial_x\bigg(\int_{\mathbb{R}}\psi(x-y)[u(x)-u(y)]\rho(y)dy\bigg)+\partial_x\bigg(u(x)\int_{\mathbb{R}}\psi(x-y)[\rho(x)-\rho(y)] dy\bigg)\\
\nonumber &=\partial_t(\mathcal{L}\rho)+\partial_x\bigg(\int_{\mathbb{R}}\psi(x-y)[u(x)-u(y)]\rho(y)dy\bigg)+\partial_x(u(x)\mathcal{L}\rho(x)).
\end{align}
Next, we apply $\partial_x$ to (\ref{euler2}) to get
\begin{equation} \label{reform2}
\partial_t(\partial_x u)+\partial_x(u\partial_x u)=\partial_x\bigg(\int_{\mathbb{R}}\psi(x-y)[u(y)-u(x)]\rho(y)dy\bigg).
\end{equation}
Thus, if we set
\begin{equation} \label{fullSys3}
G(t,x):=\partial_x u(t,x)-\mathcal{L}\rho(t,x),
\end{equation}
then, subtracting (\ref{reform1}) from (\ref{reform2}), the Euler alignment system (\ref{euler1})-(\ref{euler2}) 
can be recast into a system of equations for $\rho$ and $G$
\begin{align}
\label{fullSys1}&\partial_t \rho+\partial_x(\rho u) =0\\
\label{fullSys2}&\partial_t G+\partial_x (G u) =0.
\end{align}
The velocity field $u$ can be recovered from (\ref{fullSys3}) up to a constant. In order to find the constant,
note that the averages of $\rho$ and $G$ over $\T$ are preserved in time by (\ref{fullSys1})-(\ref{fullSys2}):
\begin{equation}\label{nov2913}
\kappa := \frac{1}{2\pi}\int_{-\pi}^{\pi} \rho(t,x) dx ~~\hbox{ and }~~\nu := \frac{1}{2\pi}\int_{-\pi}^{\pi} G(t,x) dx.
\end{equation}
Therefore, the functions
\begin{equation}\label{nov2914}
\theta(t,x) := \rho(t,x)-\kappa ~~\hbox{ and }~~\tilde{\theta}(t,x) =G(t,x)-\nu
\end{equation}
have periodic mean-zero primitive functions 
$\Phi$ and $\Psi$, respectively: 
\begin{equation}\label{nov2916}
\theta(t,x) = \d_x \Phi(t,x) ~~\hbox{ and }~~\tilde{\theta}(t,x) =\d_x \Psi(t,x).
\end{equation}
Then, $u$ can be written as 
\begin{equation}\label{udecomp}
u(t,x) = \mathcal{L} \Phi(t,x) + \Psi(t,x) + I_0(t).
\end{equation}
As in \cite{do2018global}, we find that
\begin{equation}\label{nov2917}
I_0(t) = \frac{1}{2\pi\kappa} \bigg[ \int_{-\pi}^{\pi} \rho_0(x) u_0(x) dx - \int_{-\pi}^{\pi} \rho(t,x) \Psi(t,x) dx\bigg].
\end{equation}
Note that the function $F=G/\rho$ satisfies
\begin{equation}\label{aug116}
\d_tF+u\d_xF=0,
\end{equation}
whence 
\begin{equation}\label{aug118}
\|F(t,\cdot)\|_{L^{\infty}}\le \|F_0\|_{L^{\infty}}.
\end{equation}
Combining (\ref{fullSys3}) and (\ref{fullSys1}), we have the following  equation for the density
\begin{equation}\label{main}
\d_t \rho + u\d_x \rho+ \rho \mathcal{L}\rho = -G\rho.
\end{equation}
This equation will be the starting point for our analysis below. 

\subsection{Some properties of the influence kernel}

Here, we prove some basic properties of the influence kernel that follow from our assumptions on $\psi(r)$. 
  
%
%
%
%
\begin{lemma}\label{claimdoub}
The function $M(r)$ also satisfies the doubling condition: there exists $C>0$ so that
\begin{equation}\label{aug1010}
M(r)\le CM(2r)\hbox{ for all {$r>0$}.}
\end{equation}
\end{lemma}
\begin{proof}
This is easily seen from a change of variables, using the doubling condition (\ref{doubling}) on the function $\psi(r)$: 
\begin{align*}
M(r) = \int_r^{\infty} \psi(x) dx = \frac{1}{2} \int_{2r}^{\infty} \psi(y/2) dy\leq \frac{C}{2}\int_{2r}^{\infty} \psi(y)dy = \frac{C}{2} M(2r).
\end{align*}
\end{proof}
\begin{lemma}\label{claimpow}
There exist $C_1$ and $C_2$ so that for all $k>1$, we have\begin{align}\label{power}
M(r^k)\leq C_1C_2^k [M(r)]^k~~\hbox{ for all $0<r<r_0$,}
\end{align}
with $r_0$ as in (\ref{aug1006}).
\end{lemma}
\begin{proof}
Let us define 
\[
p(y) := \log [M(e^{-y})].
\] 
There exists $y_0\ge 0$ so that the function $p(y)$ is increasing and concave for $y\ge y_0$ because 
\[
p'(y) = -\frac{M'(e^{-y})e^{-y}}{M(e^{-y})} = \frac{\psi(e^{-y})e^{-y}}{M(e^{-y})}\ge 0,
\]
and 
\[
p''(y) = -\Big( \frac{\psi(e^{-y})e^{-y}}{M(e^{-y})} \Big)' e^{-y}\leq 0,~~\hbox{ for $y\ge y_0=\log r_0^{-1}$},
\]
due to assumption (\ref{aug1006}). It follows that the function 
\[
\frac{p(y)-p(y_0)}{y-y_0}
\]
is strictly decreasing for $y\ge y_0$. Hence, for all $k>1$ we have 
\begin{equation}\label{aug1050}
\frac{p(y_0+k(y-y_0))-p(y_0)}{k(y-y_0)}\le \frac{p(y_0+y-y_0)-p(y_0)}{y-y_0},
\end{equation}
so that
\begin{equation}\label{aug1052}
p(y_0+k(y-y_0))\le kp(y)+p(y_0)-kp(y_0).
\end{equation}
Going back to the function $M$, this says
\begin{equation}\label{aug1054}
M(e^{-y_0-k(y-y_0)})\le C^{1-k}[M(e^{-y})]^k,~~C=e^{p(y_0)}=M(e^{-y_0})=M(r_0),
\end{equation}
which is
\begin{equation}\label{aug1054bis}
M\Big(\farc{r^k}{r_0^{k-1}}\Big)\le \farc{1}{[M(r_0)]^{k-1}}[M(r)]^k,~~0<r<r_0,
\end{equation}
or
\begin{equation}\label{aug1056}
M({x^k})\le \farc{1}{[M(r_0)]^{k-1}}[M(r_0^{1-1/k}x)]^k\le  \farc{1}{[M(r_0)]^{k-1}}[M(r_0x)]^k\le 
\farc{C_0^k}{[M(r_0)]^{k-1}}[M(x)]^k, 
\end{equation}
for all $0<x<r_0$.
We used the doubling property in the last inequality above.
\end{proof}

\subsection{A pointwise bound on the density }

We first obtain uniform bounds on the density $\rho(t,x)$.  
\begin{proposition}\label{lem:aug1-02} 
There exist $c_0>0$ and $C_0<+\infty$ that depend only on the initial conditions~$u_0(x)$ and $\rho_0(x)$ so that
\begin{equation}
0<c_0\leq \rho(t,x)\leq C_0, \,\,\, x\in \mathbb{T}, t\geq 0.
\end{equation}
\end{proposition}
The proof is a combination of the Constantin-Vicol maximum nonlocal principle used in~\cite{shvydkoy2017eulerian2}
in the case when $m(x)$ is singular at $x=0$ with the strategy of~\cite{do2018global}.
%
%
The function~$\Phi(t,x)$ defined in~(\ref{nov2914})-(\ref{nov2916}) satisfies a uniform bound 
\begin{equation}\label{aug124}
\|\Phi(t,\cdot)\|_{L^\infty}\le\|\theta(t,\cdot)\|_{L^1}\le C\|\rho_0\|_{L^1}.
\end{equation}
We have the following version of the Constantin-Vicol nonlocal
maximum principle. 
\begin{lemma}\label{lem:aug14}
Let $\rho(x)$ be a smooth periodic function attaining its maximum at a point $\bar x\in\T$.
There exists a positive constant $c,\tilde{c}$ such that either
\begin{align}\label{aug120}
\mathcal{L}\rho(\bar{x}) \geq c\theta(\bar{x}) M\Big(\farc{\|\Phi\|_{L^\infty}}{c\theta(\bar{x})}\Big), \,\,\, 
\text{ or }\,\,\, \theta(\bar{x})\leq \tilde{c}||\Phi||_{L^\infty}.
\end{align}
\end{lemma}
\begin{proof}
The proof is very similar to \cite{constantin2012nonlinear}. Let $\chi(x)$ be a radially non-decreasing smooth cut-off function  
such that $\chi(x) = 0$ for $|x|\leq 1/2$ and $\chi(x) =1$ for $|x|\geq 1$. We have, by the periodicity of $\rho$ and $\psi(y)$ being 
decreasing and even, for any $R\in(0,\pi)$:
\begin{align*}
\mathcal{L}\rho(\bar{x}) &= \sum_{j\in \mathbb{Z}}\int_{\mathbb{T}}(\rho(\bar{x}) - \rho(\bar{x}+y))\psi(y+2\pi j)
dy\geq \int_{\mathbb{T}}(\rho(\bar{x}) - \rho(\bar{x}+y))\psi(y)
\chi\Big(\frac{y}{R}\Big) dy  \\
&= \int_{\mathbb{T}}(\theta(\bar{x}) - \theta(\bar{x}+y))\psi(y)
\chi\Big(\farc{y}{R}\Big) dy\\
&\geq 2\theta(\bar{x})\int_{R}^{\pi} \psi(y) 
dy - \int_{\mathbb{T}} \bigl|\Phi(\bar{x}+y)\bigr| \bigl|\d_y\bigl[ \psi(y) 
\chi(\farc{y}{R})\bigr]\bigr| dy.
\end{align*}
We used integration by parts in the last integral above. The H\"ormander-Mikhlin condition~(\ref{aug106}) 
implies  
\begin{align}\label{aug1330}
\Big| \d_y\Big( \psi(y)
\chi(\farc{y}R)\Big) \Big| &=\Big|\psi'(y)\chi(\farc{y}R)+\frac{1}{R}\psi(y)\chi'(\farc{y}R)\Big|
\le C \frac{\psi(y)}{|y|}\chi(\farc{y}R)  +  \frac{\psi(y)}{R}\chi'(\frac{y}R).
\end{align}
Therefore, we get
\begin{align*}
\mathcal{L}\rho(\bar{x}) &\ge 2 \theta(\bar{x})\int_{R}^{\pi} \psi(y)ds 
-C\|\Phi\|_{L^\infty} \int_{\mathbb{T}} \Big(\frac{\psi(y)}{|y|}\chi(\frac{y}R)  +  \frac{\psi(y)}{R}\chi'(\farc{y}R)\Big)dy\\
&\ge 2\theta(\bar{x})(M(R)-M(\pi)) -  \frac{ C\|\Phi\|_{L^\infty}}{R} \int_{R/2}^{\pi} \psi(y)dy\\ 
&\ge 2\theta(\bar{x})(M(R)-M(\pi)) - C \frac{ \|\Phi\|_{L^\infty}}{R}M(R/2) 
\ge \theta(\bar{x})M(R)- C\frac{ \|\Phi\|_{L^\infty}}{R}M(R), 
\end{align*}
provided that $R$ is sufficiently small: $R<R_0$ with $R_0$ independent of the function $\rho$.  
We used (\ref{aug1010}) in the last inequality above.
If we have
\begin{equation}\label{aug112}
R_\Phi=\farc{2C\|\Phi\|_{L^\infty}}{\theta(\bar x)}<R_0,
\end{equation}
then we can take $R=R_\Phi$, leading to
\begin{align*}
\mathcal{L}\rho(\bar{x}) \geq \farc{\theta(\bar{x})}{2} M\Big(\farc{C\|\Phi\|_{L^\infty}}{\theta(\bar{x})}\Big).
\end{align*}
On the other hand, if (\ref{aug112}) does not hold, then we have
\begin{align*}
\theta(\bar{x})\leq \tilde{c}\|\Phi\|_{L^\infty}.
\end{align*}
\end{proof}

\subsubsection*{Proof of Proposition~\ref{lem:aug1-02}}
 
As $\rho$ satisfies
\begin{equation}\label{aug114}
\partial_t\rho+u\d_x\rho=-\rho\d_xu,
\end{equation}
in order to prove the upper bound $\|\rho(t,\cdot)\|_{L^\infty}\leq C_0$, with some
$C_0>\|\rho_0\|_{L^\infty}$, it is sufficient to show that 
$\d_x u(t,\bar{x}) >0$ if
$\rho(t,\bar{x})\geq C_0$. 
Moreover, we only need to consider the situation when
\begin{equation}\label{aug126}
\theta( t,\bar{x})\ge C\|\rho_0\|_{L^1},
\end{equation}
with a sufficiently large constant $C>0$ for otherwise we are done. In other words, because of~(\ref{aug124}),
we may assume that the
first alternative in (\ref{aug120}) holds. In particular, if $C>10$, it follows from (\ref{aug126}) that
\begin{equation}\label{aug128}
\frac{1}{2}\rho(t,\bar{x})\le \theta(t,\bar{x}) \le \rho(t,\bar{x})  .
\end{equation}
Using the function~$F=G/\rho$, as in~(\ref{aug116}), the $L^\infty$-bound~(\ref{aug118})
on $F$ and the  
first alternative in (\ref{aug120}), together with (\ref{aug128}), imply that
\begin{align*}
\d_x u(t,\bar{x})&= F(t,\bar{x})\rho(t,\bar{x}) + \mathcal{L} \rho(t,\bar{x})\geq 
c\theta(\bar{x}) M\Big(\farc{\|\Phi\|_{L^\infty}}{c\theta(\bar{x})}\Big) - \|F_0\|_{L^\infty} \rho(t,\bar{x})\\
&\geq \farc{c}{2}\rho(t,\bar{x})M\Big(\farc{\|\Phi\|_{L^\infty}}{c\rho(t,\bar{x})} \Big)- \|F_0\|_{L^\infty} \rho(t,\bar{x})
\end{align*}  
since $M(x)$ is a decreasing function. It follows  that $\d_x u(t,\bar{x})>0$ when $\rho(t,\bar{x})$ is large enough 
because of the uniform bound (\ref{aug124}) on $\Phi$ and  the singularity of $M(r)$ as $r\to 0$. This proves the upper bound
in Proposition~\ref{lem:aug1-02}. 

For the uniform positive lower bound, we let   $\underline{x}$ be 
a minimal point so that
\[
\rho(t,\underline{x}) = \min_{x\in \mathbb{T}} \rho(t,x),
\]
and from (\ref{main}) and $F = G/\rho$ we have 
\[
\d_t \rho + u\d_x \rho  = -F\rho^2 -  \rho \mathcal{L} \rho.
\]
Therefore, at the minimal point we have
\begin{align*}
\d_t \rho( t,\underline{x})&\geq -\|F_0\|_{L^\infty} \rho( \underline{x})^2-\rho(\underline{x})  
\sum_{j\in \mathbb{Z}}\int_{\mathbb{T}}(\rho(\underline{x}) - \rho(\underline{x}+y))\psi(y+2\pi j)dy\\ 
&\geq -\|F_0\|_{L^\infty} \rho( \underline{x})^2+\rho(\underline{x})  \int_{\mathbb{T}}(\rho(\underline{x}+y) - \rho(\underline{x}))
\psi(y)dy \\ 
& \geq -\|F_0\|_{L^\infty} \rho( \underline{x})^2+\rho(\underline{x})  \psi(\pi)
\bigl( 2\pi \kappa-2\pi \rho(\underline{x})   \bigr)\\
& = -(\|F_0\|_{L^\infty} + 2\pi \psi(\pi)) \rho( \underline{x})^2+ 2\pi\psi(\pi) \kappa \rho( \underline{x}).
\end{align*}
It follows that $\d_t \rho( t,\underline{x})>0$ if $\rho(t,\underline{x})$  is sufficiently small, thus 
there exists $c_0>0$ such that~$\rho(t,\underline{x})\geq c_0$.~$\qed$

The uniform bound (\ref{aug118}) on $F=G/\rho$ and Lemma~\ref{lem:aug1-02} give an upper bound on $G$: 
\begin{equation}\label{aug130}
|G(t,x)| \le C |\rho(t,x)|\le \tilde{C}.
\end{equation}
The function $Q = \d_x F/\rho$ also satisfies the transport equation
\begin{equation}\label{aug131}
\d_t Q +u\d_x Q = 0,
\end{equation}
which gives the bounds
\begin{equation}\label{aug132}
\begin{aligned}
 &|\d_x F(t,x)| \le C |\rho(t,x)|\le C,\\
 &|\d_x G(t,x)|\leq |\d_x F|\rho + |F||\d_x \rho|\le C[ \rho(t,x)^2 + |\rho_x(t,x)|]\le C (1+|\rho_x(t,x)|).
 \end{aligned}
 \end{equation}
The arguments leading to (\ref{aug130})-(\ref{aug132}) can be iterated to obtain a higher order control of $G$: 
the function $\d_x Q/\rho$ satisfies the transport equation, and so on, 
leading to the hierarchical point-wise bounds as in \cite{shvydkoy2017eulerian1}, \cite{shvydkoy2017eulerian3}:
 \begin{equation}\label{dec510}
|G^{(n)}(t,x)| \le C_n( |\rho^{(n)}(t,x)| +  |\rho^{(n-1)}(t,x)| +\cdots + |\rho(t,x)|). 
\end{equation}
 
\section{The proof of Theorem~\ref{thm-holder}} 

 \subsection{A H\"older regularity result for a class of linear integro-differential equations}
 
We first investigate the H\"{o}lder estimates for solutions to a class of integro-differntial equations  of the form
\begin{align}\label{homogmain}
\d_t q + v(t,x)\cdot\nabla q + \mathscr{L} q = 0,
\end{align}
in $\R^d$ with an operator $\mathscr L$ given by
\begin{equation}\label{aug602}
\mathscr{L} q(t,x) = \int_{\mathbb{R}^d} (q(t,x)-q(t,x+z)) k(x,z,t) dz,
\end{equation}
and a kernel $k(x,z,t)$ such that there exist $C>0$ and a function $\eta(r)$ such that
\begin{equation}\label{aug604}
\frac{1}{C}\eta(|z|)\le k(x,z,t)\le C\eta(|z|).
\end{equation}
We assume here  $\eta(r)$ satisfies the following properties, as in our assumptions on
$\psi(r)$. First, we suppose that for any $\alpha\in(0,1)$, there exists $c_{\alpha}$ such that
\[
\eta(r)\ge \frac{1}{c_\alpha r^{d-\alpha}} \hbox{ for all $0<r\le 1$}, 
\]
and 
\begin{equation}\label{aug606}
M_\eta(r):=\int_r^\infty \eta(s)s^{d-1} ds\to +\infty\hbox { as $r\to 0$}.
\end{equation}
We also assume that the function $M_\eta(r)$ satisfies  (\ref{aug1006}) in the $d$-dimensional
form
\begin{equation}\label{jan206}
\frac{r^d\eta(r)}{M_\eta(r)}\hbox{ is non-decreasing for $0\le r\le r_0$},
\end{equation}
and (\ref{aug1008}), which, as we recall, implies
\begin{align}\label{aug608}
s{M_\eta(s)} \le r M_\eta(r) \,\,\, \text{ for } 0<s\leq r\leq r_0.
\end{align}
We also assume that the drift $v(t,x)$ grows at most linearly at infinity, and is
$M_\eta$-Lipschitz continuous in $x$ for each~$t>0$, in the sense that there 
exist $C, C'>0$ such that
\begin{eqnarray}
&& \farc{|v(t,x)|}{1+|x|}\le C,~~\hbox{ for all $0\le t$, and $x\in\R$,}\label{uest0}\\
&&|v(t,x) - v(t,y)|\le C' |x-y| M_\eta(|x-y|) \text{ for  $t\ge 0$ and } |x-y|\leq 1.\label{uest}
\end{eqnarray}
Our goal in this section is to show the following by an analytical approach.
\begin{proposition}  \label{holderthm}
Assume that (\ref{aug604})-(\ref{uest}) hold, and let $q(t,x)$ be 
a solution to (\ref{homogmain}).
There exists $\beta \in (0,1)$, $r_0>0$ and a sufficiently small constant $c'>0$ such that for any $0<t\leq 1$ 
and $|x-y| \leq \min(r_0,M^{-1}_\eta(c'/t))$, we have
\begin{align}\label{Mholder}
|q(t,x) - q(t,y)|\le K t^{-{\beta}} [M_\eta(|x-y|)]^{-\beta}
\sup_{0\le t\le 1,x\in\R^d}|q(t,x)|.
\end{align}
\end{proposition}
The first step in the proof of Proposition~\ref{holderthm} is to re-center:
fix $(t_0,x_0)$ with~$0<t_0<1$ and $x_0\in\Rm^d$, and write
\begin{align*}
\tilde{q}(t,x) = q(t_0+t,x_0+x-A(t))-q(t_0,x_0)
\end{align*}
with the function $A(t)$ to be determined. Note that $\tilde q(0,0)=0$,
as long as $A(0)=0$. The function $\tilde q(t,x)$ satisfies
\begin{align}\label{nov1502}
\d_t\tilde q + \tilde v(t,x)\cdot\nabla\tilde q + \widetilde{\mathscr L}\tilde q = 0,
\end{align}
with
\begin{equation}\label{nov1506}
\tilde v(t,x)=v(t+t_0,x+x_0-A(t))+A'(t),
\end{equation}
and
\begin{equation}\label{nov1504}
\begin{aligned}
\widetilde {\mathscr L}\tilde q(t,x)=
\int_{\Rm^d}(\tilde q(t,x)-\tilde q(t,x+z))\tilde k(t,x,z)dz
\end{aligned}
\end{equation}
where
\[
\tilde k(t,x,z)=k(t+t_0,x+x_0-A(t),z).
\]
Note that the function $\tilde k(x,z,t)$ still satisfies the bounds (\ref{aug604})
by $\eta(z)$ from above and below.
We choose $A(t)$ as a 
solution to
\begin{equation}\label{nov1508}
A'(t)=-v(t_0+t, x_0-A(t)),~~A(0) = 0,
\end{equation}
so that $\tilde v(t,0)=0$. 
A solution to (\ref{nov1508}) exists due to the continuity of $v$ in $x$. 
Note that by the $M_\eta$-Lipschitz continuity (\ref{uest}) 
of $v(t,x)$ in $x$ and the choice of $A(t)$ we have 
\begin{equation}\label{nov1514}
|\tilde{v}(t,x)|=|v(t+t_0,x+x_0-A(t)) -v(t_0+t, x_0-A(t))|
\leq C'|x| M_\eta (|x|),~~|x|\le 1,
\end{equation}
and because of (\ref{uest0}), $\tilde v(t,x)$ is sublinear at infinity:
\begin{equation}\label{nov1516}
|\tilde v(t,x)|\le C_0(1+|x|), 
\end{equation}
with a constant $C_0$ that depends on $x_0$ and $t_0$.

To use a De Giorgi-type argument, given $r>0$, we define an $M_\eta$-parabolic cylinder
\begin{align}\label{cyl}
Q(r):=  (-c'/M_\eta(r), 0]\times B(r).
\end{align}
Here, $B(r) : = \{x\in \mathbb{R}^d: |x|<r \}$ is a ball in $\Rm^d$, and $c'>0$ is sufficiently small to be chosen later.
Proposition~\ref{holderthm} is a consequence of the following lemma.
\begin{lemma}\label{lem-nov1504}
There exist a sufficiently large constant $K>0$ and a sufficiently small constant~$c'>0$
that do not depend on $(t_0,x_0)$, and $r_0>0$ so that for
all~$0<r<r_0$ we have
\begin{align}\label{anotherver}
|\tilde{q}(t,x)| 
\leq K[M_\eta(r)]^{\beta}\Big(\frac{|t|}{c'} +\frac{1}{[M_\eta(|x|)]} \Big)^{\beta}
\sup_{-c'/M_\eta(r) \le t\le 0,x\in\R^d}|\tilde{q}(t,x)|,~~~\hbox{for all $(t,x)\in Q(r)$.}
\end{align}
\end{lemma}
In terms of the function $q(t,x)$,   (\ref{anotherver})
says that
\begin{equation}\label{nov1520}
|q(t+t_0,x+x_0-A(t))-q(t_0,x_0)|\le 
2K[M_\eta(r)]^{\beta}\Big(\frac{|t|}{c'}+\frac{1}{[M_\eta(|x|)]}\Big)^{\beta}
\sup_{t_0-c'/M_\eta(r) \le t\le t_0,x\in\R^d}|{q}(t,x)|,
\end{equation}
for $(t,x)\in Q(r)$. Taking $r=M^{-1}_\eta(c'/|t_0|)$ 
and setting $t' = t_0+t$ and $x' = x_0+x-A(t)$
in~(\ref{nov1520}) gives
\begin{equation}\label{nov1522}
|q(t',x')-q(t_0,x_0)|\le 
\farc{2K}{t_0^{\beta}}\Big( \frac{|t'-t_0|}{c'}+\frac{1}{M_\eta(|x'-x_0+A(t)|)}\Big)^{\beta}
\sup_{0 \le t\le t_0,x\in\R^d}|{q}(t,x)|,
\end{equation} 
for $0\le t'\le t_0$, and $|x'-x_0+A(t)|\le M^{-1}_\eta(c'/|t_0|)$. 
It follows from (\ref{nov1516}), 
that there exists~$\lambda_0>0$ which depends on $(t_0,x_0)$ such that 
\begin{align*}
|A(t)|\leq \lambda(e^{C_0|t|}-1 ),
\end{align*}
with $C_0$ as in (\ref{nov1516}), thus
\begin{equation}\label{nov1524}
|q(t', x') - q(t_0,x_0)| \leq \farc{2K}{ |t_0|^{\beta}}
\bigg(\frac{|t'-t_0 |}{c'}+ \frac{1}{M_\eta(|x'-x_0|+
\lambda_0(e^{C_0|t'-t_0|}-1 ))}\bigg)^{\beta} \sup_{0 \le t\le t_0,x\in\R^d}|q(t,x)|,
\end{equation}
for $t' \in (0, t_0]$ and 
\begin{equation}\label{nov1526}
|x'-x_0| +\lambda_0(e^{C_0|t'-t_0|}-1 )\le M^{-1}_\eta(c'/|t_0|).
\end{equation}
Setting $t'=t_0$ in (\ref{nov1524}) and (\ref{nov1526}) finishes the proof
of Proposition~\ref{holderthm}. Note that
both constants $C_0$ and $\lambda_0$ that may depend on $(t_0,x_0)$ disappear 
when $t'=t_0$. Thus, we only need to prove Lemma~\ref{lem-nov1504}.

\subsubsection*{The proof of Lemma~\ref{lem-nov1504}}

The proof uses a De Giorgi type argument. We fix $r>0$, 
and normalize  $\tilde q(t,x)$, setting 
\begin{align*}
w(t,x) = \frac{\tilde{q}(t,x)}{2||\tilde{q}(t,x)||_{L^{\infty}([-c'/M_\eta(r),0]\times \R^d)}}.
\end{align*}
We also define a decreasing sequence of radii $r_n$ as
\begin{align}\label{radius}
r_n = M^{-1}_\eta (a^{n-1} M_\eta(r)),
\end{align}
with some $a>2$ to be specified later. Note that
\[
M_\eta(r_n)=a M_\eta(r_{n-1}),
\]
thus $M_\eta(r_n)\to+\infty$, and $r_n\to 0$, as $n\to\infty$.
\begin{lemma}\label{lem-nov1502}
There exists $b\in(1,2)$ and $n_0\in\N$ so that for all $n\ge n_0$ we have
\begin{align}\label{osc}
\osc_{Q(r_n)} w(t,x) \le b^{1-n}.
\end{align}
\end{lemma}
Lemma~\ref{lem-nov1504} is an immediate consequence of Lemma~\ref{lem-nov1502}.
As $r_n\to 0$, for any $(t,x) \in Q(r)$, there exists $n\in \N$ such 
that $(t,x)\in Q(r_n) \setminus Q(r_{n+1})$, so that 
\begin{align*}
\text{ either } \frac{c'}{M_\eta(r_{n+1})} < |t| \le 
\frac{c'}{M_\eta(r_{n})} \,\,\, \text{ or }\,\,\, r_{n+1}\le|x|<r_n.
\end{align*}
Thus we have
\begin{align*}
|w(t,x)| &= |w(t,x) -w(0,0)| \leq \osc_{Q(r_n)} w(t,x) \leq b^{1-n}\\
& = b\bigg(\frac{1}{a^n}\bigg)^{\frac{\log b}{\log a}} = b \bigg(
\frac{M_\eta(r)}{M_\eta(r_{n+1})}\bigg)^{\frac{\log b}{\log a}} \leq b [M_\eta(r)]^{\frac{\log b}{\log a}}  \bigg(\frac{|t|}{c'} +
\frac{1}{M_\eta(|x|)}\bigg)^{\frac{\log b}{\log a}},
\end{align*}
thus (\ref{anotherver}) holds with $\beta = \log b/\log a$, finishing the proof of
Lemma~\ref{lem-nov1504}. 

\subsubsection*{The proof of Lemma~\ref{lem-nov1502}}

We prove (\ref{osc}) by induction. For $n=1$, it holds automatically since
$|w(t,x)|\le 1/2$ for all~$(t,x)\in Q(r)$. 
Suppose that (\ref{osc}) holds for $1\le k\le n$ and set
\begin{align*}
m := \frac{1}{2}\big(\sup_{Q(r_n)} w(t,x) + \inf_{Q(r_n)} w(t,x)\big),
\end{align*}
so that the function
\begin{align*}
\tilde{w}(t,x):= 2b^{n-1}(w(t,x)-m),
\end{align*}
satisfies $|\tilde{w}(t,x)|\le 1$ for $(t,x)\in Q(r_n)$. 
%
It is convenient to set
\begin{equation}\label{nov2002}
\varphi(r) = M^{-1}_\eta(a^{-1}M_\eta(r)),
\end{equation}
so that $r_{n} = \varphi(r_{n+1})$, and a measure  
\begin{align*}
\mu(dy) = \frac{\eta(|y|)}{M_\eta(|y|)} dy.
\end{align*}
To proceed with the induction argument, 
we need to find $a>2$, so that
\begin{equation*}
\begin{aligned}
(-c'/M_\eta(r_n), -c'/M_\eta(r_{n+1})]& = 
(-ac'/M_\eta(r_{n+1}), -c'/M_\eta(r_{n+1})]\\
&\supset (-2c'/M_\eta(r_{n+1}), -c'/M_\eta(r_{n+1})],  
\end{aligned} 
\end{equation*} 
and $\theta \in (0,1)$, so that if~$|\tilde w(t,x)|\le 1$
on $Q(r_n)$, and
\begin{equation}\label{nov1530} 
\begin{aligned}
&\mu\Big((t,x)\in (-2c'/M_\eta(r_{n+1}), -c'/M_\eta(r_{n+1})]
\times [B(\varphi(r_{n+1}))\setminus B(r_{n+1})]:~\tilde{w}(t,x)\leq 0\Big) \\
&\geq \frac{1}{2} \mu \Big((-2c'/M_\eta(r_{n+1}), -c'/M_\eta(r_{n+1})]
\times [B(\varphi(r_{n+1}))\setminus B(r_{n+1})  ] \Big),
\end{aligned}
\end{equation}
then we have $\tilde{w} \leq 1-\theta$ on $Q(r_{n+1})$.
This requires the following lemma on lowering the maximum. 
\begin{lemma}\label{pointest}
Fix $0<r<r_0$, with $M_\eta(r_0) \geq 1$, and fix $\delta>0$. Assume that
$w(t,x)\leq 1$ for all~$(t,x)\in Q(\varphi(r))$ and the function
$w(t,x)$ satisfies 
\begin{align*}
\d_t w(t,x) + v(t,x) \cdot\nabla w(t,x) + \mathscr{L} w(t,x) = 0,~~(t,x)\in Q(\varphi(r)),
\end{align*}
with 
\begin{equation}\label{nov1554}
|v(t,x)| \leq C' |x|M_\eta(|x|)\hbox{ for $|x|\le 1$},
\end{equation}
 and the operator
$\mathscr{L}$ as in (\ref{aug602})-(\ref{aug608}).
There exists a constant $c'>0$ sufficiently small, and another constant $\theta\in(0,1)$, so that if
\begin{align}\label{nov1548}
\mu\bigg(\{w(t,x)\leq 0\}\cap \big((-2c'/M_\eta (r), -c'/M_\eta(r)] \times B(\varphi(r))\setminus B(r)\big)
 \bigg)\geq \delta >0,
\end{align} 
then
\begin{equation}\label{nov1550}
\hbox{$w\leq 1-\theta $ in $Q(r)$.}
\end{equation}
\end{lemma}

The conclusion of Lemma~\ref{lem-nov1502} follows from Lemma \ref{pointest}.
Indeed, this lemma implies that if (\ref{nov1530}) holds, then
we have $\tilde{w}(t,x)\leq 1-\theta$ on $Q(r_{n+1})$, so that
\begin{align*}
w(t,x) \leq \frac{1-\theta}{2} b^{1-n} + m \,\,\, \text{ for } (t,x)\in Q(r_{n+1}).
\end{align*}
As $\inf_{Q(r_{n+1})} w(t,x) \geq \inf_{Q(r_{n})} w(t,x) $, then
the oscillation of $w(t,x)$ on $Q(r_{n+1})$ is bounded~by
\begin{align*}
\osc_{Q(r_{n+1})} &w(t,x) \leq \frac{1-\theta}{2} b^{1-n} + m-\inf_{Q(r_{n+1})} w(t,x)\\
&\leq \frac{1-\theta}{2} b^{1-n} + m - \inf_{Q(r_{n})} w(t,x)= \frac{1-\theta}{2} b^{1-n} + \frac{1}{2}(\sup_{Q(r_n)} w(t,x) - \inf_{Q(r_n)}w(t,x)) \\
&=  \frac{1-\theta}{2} b^{1-n} + \frac{1}{2}\osc_{Q(r_n)} w(t,x) \leq
\frac{1-\theta}{2} b^{1-n} + \frac{1}{2}b^{1-n}
= \frac{2-\theta}{2} b^{1-n} = b^{-n},
\end{align*}
if we choose 
\begin{equation}\label{nov1532}
b= \frac{2}{2-\theta}\in(1,2).
\end{equation}
We used the induction assumption (\ref{osc})
in the last inequality above. If (\ref{nov1530}) does not hold, then we have 
\begin{align*}
\mu\bigg(\{\tilde{w}(t,x)> 0\}\cap & \big((-2c'/M_\eta (r_{n+1}), -c'/M_\eta(r_{n+1})] \times [B(\varphi(r_{n+1}))\setminus B(r_{n+1})]\big) \bigg) \\
&\geq \frac{1}{2} \mu \big((-2c'/M_\eta (r_{n+1}), -c'/M_\eta(r_{n+1})] \times [B(\varphi(r_{n+1}))\setminus B(r_{n+1})]\big),
\end{align*}
and we can repeat the argument above for $-\tilde{w}(t,x)$. This finishes the proof
of Lemma~\ref{lem-nov1502}. 

\subsubsection*{Proof of Lemma \ref{pointest}}

Let $\varrho(r)$, $r>0$, be a radially smooth non-increasing function such 
that $\varrho(r)>0$ for $0\le r<2$, with~$\varrho(r) =1 $ for $0\le r\le 1$, and~$\varrho(r) =0$ for $r\geq 2$. 
We set 
\[
\varrho_r(t,x) = \varrho\Big(\frac{|x|+C_v  r M_\eta(r) t}{r}\Big) =
\varrho\Big(\frac{|x|-C_v r M_\eta(r) |t|}{r}\Big),
\]
with a large constant $C_v$ to be chosen later.
Condition (\ref{aug608}) implies that
\begin{equation*}
\frac{M_\eta(2r+2C_v c' r  )}{M_\eta(r)}\ge \frac{1}{2(1+C_v c' )}\ge  \frac{1}{a} 
\end{equation*}
if we choose $a$ sufficiently large. 
As
\[
M_\eta(\varphi(r)) = a^{-1}M_\eta(r),
\]
it follows that 
\[
\varphi(r) \ge 2r+ C_v r M_\eta(r)|t|~~
\hbox{for $- 2c'/M_\eta(r)\le t\le 0$}.
\] 
We also set
\begin{align}\label{nov1544}
\mu(t)=\mu \big(\{w(t,x)\leq 0\}\cap (B(\varphi(r))\setminus B(r)) \big),
\end{align}
and
\begin{align}\label{nov1545}
\bar\mu_r=\mu\bigg(\{w\leq 0\}\cap \big((-2c'/M_\eta (r), -c'/M_\eta(r)] \times B(\varphi(r))\setminus B(r)
\big ) \bigg)\ge\delta,
\end{align}
by (\ref{nov1548}). 

Our goal will be to show that 
\begin{equation}\label{nov1540}
w(t,x)\le 1-\zeta(t)\varrho_r(t,x) \hbox{ in 
$Q(r)$}.
\end{equation}
The function $\zeta(t)$ in (\ref{nov1540}) obeys the ODE
\begin{align}\label{nov1542}
\zeta'(t) = \sigma \mu(t)
- C_1 M_\eta(r) \zeta(t),~~ -2c'/M_\eta(r)\le t\le  0,
\end{align}
with the initial condition
\[
\zeta\Big(-\frac{2c'}{M_\eta(r)}\Big) = 0.
\]
To ensure that $\zeta(t)$ is $C^1$ on 
$(-ac'/M_\eta(r), 0]$, we extend it to $t\le -2c'/M_\eta(r)$ so that 
\[
\lim_{t\to (-2c'/M_\eta(r))^-}\zeta'(t) = \sigma \mu(-2c'/M_\eta(r)),
\]
and $\zeta(t)\le 0$ for $t\le -2c'/M_\eta(r)$. The solution to (\ref{nov1542}) is
\begin{align*}
\zeta(t) = \int_{-2c'/M_\eta(r)}^t  \sigma
\mu(s)
e^{-C_1 M_\eta(r) (t-s)} ds.
\end{align*}
Hence,  we have a lower bound
\begin{align}\label{nov1546}
\zeta(t)& \geq \sigma e^{-2 c' C_1 } \int _{-2c'/M_\eta(r)}^t \mu(s)ds\ge 
\sigma e^{-2 c' C_1 } \bar\mu_r
\geq \sigma e^{-2c'C_1} \delta,
\end{align}
for all $t\in [-c'/M_\eta(r),0]$. 
Going back to (\ref{nov1540}), it follows that
\begin{equation}\label{nov1552}
w(t,x)\le 1-\sigma e^{-2 c' C_1 } \delta\hbox{ in $Q(r)$},
\end{equation}
thus (\ref{nov1550}) holds with $\theta = \sigma e^{-2c' C_1 } \delta$ if we choose
a small $\sigma$ and a sufficiently large $C_1$.

The proof of (\ref{nov1540}) is by contradiction. Suppose that 
\begin{equation}\label{nov2004}
w(t,x) > 1- \zeta(t)\varrho_r(t,x)
\end{equation}
at some point $(t,x)\in Q(r)$. Let $(t_0,x_0)$ be the maximal point of the function 
\[
w(t,x) + \zeta(t)\varrho_r(t,x),
\] 
so that, in particular, 
\begin{equation}\label{nov2004bis}
w(t_0,x_0)+\zeta(t_0)\varrho_r(t_0,x_0) >1.
\end{equation}
As a consequence of (\ref{nov2004bis}) and the assumption that
$w(t,x)\le 1$, $t_0$ must be in the time interval where $\zeta(t)>0$ and $(t_0, x_0)$ must be in the support of $\varrho_r$, hence $t_0\in (-2c'/M_\eta(r), 0]$ and 
\begin{equation}\label{nov2006}
|x_0| <2r(1+C_v c' ).
\end{equation}

Thus, at $(t_0, x_0)$ we have
\begin{align*}
& \d_t w(t_0,x_0) + \zeta'(t_0) \varrho_r(t_0,x_0) + \zeta(t_0) \d_t\varrho_r(t_0,x_0) \geq 0, \\
& \nabla w(t_0,x_0) +\zeta(t_0) \nabla \varrho_r(t_0,x_0) = 0.
\end{align*}
This gives a lower bound
\begin{equation}
\begin{aligned}\label{contraest}
0 &= \d_t w(t_0,x_0) + v(t_0,x_0) \cdot\nabla w(t_0,x_0) + \mathscr{L} w(t_0,x_0) \\
& \geq -\zeta'(t_0) \varrho_r(t_0,x_0) + C_v r M_\eta(r)\zeta(t_0) |\nabla \varrho_r(t_0,x_0)| - \zeta(t_0) v(t_0,x_0) \cdot
  \nabla \varrho_r(t_0, x_0)+\mathscr{L} w(t_0,x_0)\\
  &\geq  -\zeta'(t_0) \varrho_r(t_0, x_0)+ \big(C_v r M_\eta(r)-C'|x_0|M_\eta(|x_0|)\big)\zeta(t_0) |\nabla \varrho_r(t_0,x_0)| +\mathscr{L} w(t_0,x_0).
\end{aligned}
\end{equation}
When $|x_0|<r$, then $(t_0,x_0)$ satisfies
\begin{align}
\Big| |x_0| - C_v r M_\eta(r) |t_0|\Big| <r,
\end{align}
given $C_vr M_\eta(r) |t_0| < r$, which requires
\begin{align}\label{nov2614}
C_v c' <1/2.
\end{align}
Condition (\ref{nov2614})   holds if we pick $c'$ small. With that, whenever $|x_0| <r$, $\varrho_r(t_0,x_0) = 1$ and thus the term in (\ref{contraest}) with $\nabla \varrho_r(t_0,x_0) = 0 $ disappears. So we may consider the more difficult case $|x_0|\ge r$. Combining (\ref{nov2006}) and $M_\eta(|x_0|) \le M_\eta(r)$, we have
\[
|x_0| M_\eta(|x_0| ) \le 2r(1+C_v c' ) M_\eta(r).
\]
Therefore, it gives
\begin{equation*}
\begin{aligned}
&C_v r M_\eta(r)-C'|x_0|M_\eta(|x_0|)\ge
C_vrM_\eta(r) - 2rC'{(1+C_v c')}  M_\eta(r ) \\
&=C_vr M_\eta(r)\Big[1- \farc{2C' (1+C_v c')}{C_v}\Big]\ge 0,
\end{aligned}
\end{equation*}
if we take $C_v\ge 4C'$ and  $c' < 1/(8C')$, so that
\begin{equation}\label{nov2612}
\farc{2C' (1+C_v c')}{C_v}=  \farc{2C'}{C_v}+{2C'c'}<1.
\end{equation}
Note that we can ensure that (\ref{nov2612}) holds while keeping (\ref{nov2614}) intact, if we first choose $C_v$ large and then $c'$ small.

Now, no matter where $x_0$ locates, (\ref{contraest}) becomes
\begin{equation}
\begin{aligned}\label{contraest-bis}
0 &= \d_t w(t_0,x_0) + v(t_0,x_0) \cdot\nabla w(t_0,x_0) + \mathscr{L} w(t_0,x_0) 
\geq -\zeta'(t_0) \varrho_r(t_0, x_0)+\mathscr{L} w(t_0,x_0).
\end{aligned}
\end{equation}
To get a contradiction, we will need a lower bound on $\mathscr{L} w(t_0,x_0)$ in the right side of (\ref{contraest-bis}).
First, we write
\[
\begin{aligned}
&\mathscr{L} \varrho_r(t,x)=\int_{\Rm^d}[\varrho_r(t,x)-\varrho_r(t,x+z)]k(x,z,t)dz\\
&=\int_{\Rm^d}\Big[\varrho\Big(\frac{|x|-C_v r M_\eta(r) |t|}{r}\Big)- 
\varrho\Big(\frac{|x+z|-C_v r M_\eta(r) |t|}{r}\Big)\Big]k(x,z,t)dz,
\end{aligned}
\]
hence
\[
\begin{aligned}
&\mathscr{L} \varrho_r(t,ry)=\int_{\Rm^d}[\varrho_r(t,ry)-\varrho_r(t,ry+z)]k(ry,z,t)dz\\
&=r^{-d}\int_{\Rm^d}\Big[\varrho\Big(\frac{|ry|-C_v r M_\eta(r) |t|}{r}\Big)- 
\varrho\Big(\frac{|ry+rz|-C_v r M_\eta(r) |t|}{r}\Big)\Big]k(ry,rz,t)dz,
\\
&=r^{-d}\int_{\Rm^d}\Big[\varrho(|y|-C_v M_\eta(r) |t|)- 
\varrho(|y+z|-C_v  M_\eta(r) |t|)\Big]k(ry,rz,t)dz.
\end{aligned}
\]
Let us set $\tilde \varrho(t,x)=\varrho(|x|-C_v  M_\eta(r) |t|)$, so that 
\begin{equation}\label{nov2616}
\begin{aligned}
&\mathscr{L} \varrho_r(t,ry)&=r^{-d}\int_{\Rm^d}\Big[\tilde\varrho(t,|y|)- 
\tilde\varrho(t,|y+z|)\Big]k(ry,rz,t)dz.
\end{aligned}
\end{equation}
Note that if $\varrho_r(t,ry)=0$, then $\tilde\varrho(t,y)=0$, and (\ref{nov2616}) becomes
\begin{equation}\label{nov2618}
\begin{aligned}
&\mathscr{L} \varrho_r(t,ry)=-r^{-d}\int_{\Rm^d}
\tilde\varrho(t,|y+z|)k(ry,rz,t)dz\le
-Cr^{-d}\int_{\Rm^d}
\tilde\varrho(t,|y+z|)\eta(rz)dz \\
& \le -Cr^{-d}\int_{\Rm^d}
\tilde\varrho(t,|y+z|)\frac{1}{c_{\alpha} |rz|^{d-\alpha}}dz = -\frac{C}{c_{\alpha}} r^{-2d+\alpha}
\int_{\Rm^d}\varrho(|y+z|-C_v  M_\eta(r) |t|) |z|^{\alpha-d} dz.
\end{aligned}
\end{equation}
 {Due to $|t|<2c'/M_\eta(r)$ and (\ref{nov2614}), the time dependent shift  in $\varrho$ is of $O(1)$. Because $\tilde{\varrho}$ is compactly supported, we may pick sufficiently large $a$ to make $|rz|<1$ guaranteed, so that we can apply the lower bound of $\eta$ to factor $r$ out from the integral. Observe that when~$t=0$ and $|y|=2$, the right side of (\ref{nov2618}) is strictly negative. It follows that there exists a universal constant $c_1>0$ so that we still have  
\begin{equation}\label{nov2620}
\hbox{$\mathscr{L} \varrho_r (t,x)\le 0$ provided that
$\varrho_r(t,x)\le c_1$. }
\end{equation}
}
We will consider two cases for a lower bound on $\mathscr{L} w(t_0,x_0)$: 
if $\varrho_r (t_0,x_0)> c_1$, then we will show that
\begin{align}\label{nov1564}
\mathscr{L} w(t_0,x_0) \geq -C\zeta(t_0) M_\eta(r) + \sigma M_\eta(r)\mu(t_0).
\end{align}
Using this estimate, together with the ODE (\ref{nov1542}) for $\zeta'(t)$  in (\ref{contraest-bis}) gives
\begin{align}\label{nov1566}
0 \geq \sigma \mu(t_0)
\big(M_\eta(r)-\varrho_r(t_0,x_0)\big) + M_\eta(r) \zeta(t_0) \big(C_1\varrho_r(t_0, x_0)-C\big) .
\end{align}
The condition 
\begin{align}\label{nov1841}
M_\eta(r)\geq M_\eta(r_0)\geq 1
\end{align}
in Lemma~\ref{pointest} 
ensures that the first term in the right side of (\ref{nov1566}) 
is non-negative. Moreover, since $\varrho_r(t_0,x_0)> c_1$ and $\zeta(t_0)>0$, we get a contradiction if we choose~$C_1$ 
in (\ref{nov1542}) large enough. 
 
When $\varrho_r(t_0,x_0)\leq c_1$, we will obtain the following lower bound for 
$\mathscr{L} w(t_0,x_0)$: 
\begin{align}\label{nov1843}
    \mathscr{L}w(t_0,x_0) \geq  \sigma M_\eta(r)\mu(t_0).
\end{align}
Again, using this estimate, together with  (\ref{nov1542}) in (\ref{contraest-bis}) gives
\begin{align}
   0 \geq \sigma \mu(t_0) \big(M_\eta(r)-\varrho_r(t_0,x_0)\big) + M_\eta(r) \zeta(t_0) C_1\varrho_r(t_0, x_0).
\end{align}
This is a contradiction to  (\ref{nov1841}) and $\varrho_r(t_0,x_0)<1$  for any $C_1\geq 0$.

It remains to show that estimates (\ref{nov1564}) and  (\ref{nov1843}) hold in their domains of validity. 
We will drop $t_0$ in $\varrho_r(t_0,x_0)$ as it does not affect the following computation.
Since the function~$w+\zeta \varrho_r$ obtains its maximum at $x_0$, we have 
\[
w(t_0, x_0) - w(t_0, x_0+y) \geq - \zeta(t_0) (\varrho_r (x_0) - \varrho_r (x_0+y)),
\]
for all $y\in\Rm^d$.  Note that if
$w(t_0,x_0+z) \leq 0$, then, because of (\ref{nov2004bis}), we have 
\begin{align}\label{nov1570}
w(t_0,x_0) + \zeta(t_0) \varrho_r(x_0) - w(t_0,x_0+z) - 
\zeta(t_0) \varrho_r(x_0+z) \geq 1-\zeta(t_0) \geq \frac{1}{2},
\end{align}
if we choose $c_0$ in (\ref{nov1542}) to be sufficiently small. Let us
introduce the ``good" set 
\[
G:= \{w\leq 0\}\cap (B(\varphi(r))\setminus B(r)),
\] 
and write
\begin{equation}\label{nov1588}
\begin{aligned}
\mathscr{L} w(t_0,x_0) &= \int_{\R^d} ( w(t_0,x_0) -w(t_0,x_0+y))k(x_0,y,t_0) dy \\
&=\int_{\R^d} ( w(t_0,x_0)+\zeta(t_0)[\varrho_r(x_0)-\varrho_r(x_0+y)] -w(t_0,x_0+y))k(x_0,y,t_0) dy\\
&-\zeta(t_0)\int_{\R^d}[\varrho_r(x_0)-\varrho_r(x_0+y)]k(x_0,y,t_0) dy
\\
&=\int_{x_0+y\in G} ( w(t_0,x_0)+\zeta(t_0)\varrho_r(x_0)-\zeta(t_0)\varrho_r(x_0+y) -w(t_0,x_0+y))k(x_0,y,t_0) dy\\
&+\int_{x_0+y\not\in G} ( w(t_0,x_0)+\zeta(t_0)\varrho_r(x_0)-\zeta(t_0)\varrho_r(x_0+y) -w(t_0,x_0+y))k(x_0,y,t_0) dy\\
&-\zeta(t_0)\int_{\R^d}[\varrho_r(x_0)-\varrho_r(x_0+y)]k(x_0,y,t_0) dy\\
&\geq -\zeta(t_0)\int_{\R^d} (\varrho_r(x_0) - \varrho_r (x_0+y)) k(x_0,y,t_0) dy\\
&+\int_{x_0+y\in G} (w(t_0,x_0) + \zeta(t_0) \varrho_r(x_0) - w(t_0,x_0+y) - 
\zeta(t_0) \varrho_r(x_0+y) )k(x_0,y,t_0) dy\\
&\geq -\zeta(t_0)\int_{\R^d} [\varrho_r(x_0) - \varrho_r(x_0+y)]k(x_0,y,t_0) dy + 
\frac{1}{2C} \int_{x_0+y\in G} \eta(y) dy =: J_{1} + J_{2}.
\end{aligned}
\end{equation}
We used (\ref{nov1570}) in the last two steps above.
To bound $J_1$, when $\varrho_r(t_0,x_0)>c_1$ we write
\begin{equation}\label{nov1574}
\begin{aligned}
J_{1} &\geq -C\zeta(t_0) \bigg(  \int_{|y|\leq r} 
\|\nabla \varrho_r\|_{L^{\infty}} |y| \eta(y) dy  + 
2\|\varrho_r\|_{L^{\infty}}\int_{|y|\geq r} \eta(y) dy\bigg)\\
&\geq -C\zeta(t_0)\bigg(\frac{1}{r}  \int_{|y|\leq r} |y| \eta(y) dy + M_\eta(r) \bigg) \geq -C\zeta(t_0) M_\eta(r),
\end{aligned}
\end{equation}
because assumption (\ref{jan206}) implies that there exists $C$ such that 
\begin{equation}\label{jan204}
\hbox{$|y|^d\eta(y) \le C M_\eta(|y|)$ for all $0\le |y|\le r\le r_0$,}
\end{equation}
and we can take $\gamma\in (0,1/2]$ as in the the assumption (\ref{aug1008}) to have $|y|^{\gamma } M_\eta(|y|)$ be an increasing function for $0\le|y|<r_0$, so that the first term in the last line satisfies
\begin{equation}\label{nov2813}
\begin{aligned}
    \int_{|y|\leq r} |y| \eta(y) dy &\le C \int_{|y|\leq r} \frac{M_\eta(|y|)}{|y|^{d-1}} dy = C\int_{|y|\leq r} \frac{|y|^{\gamma } M_\eta(|y|)}{|y|^{\gamma +d -1}} dy \\
    &\le C r^{\gamma } M_\eta(r) \int_{|y|\leq r} \frac{1}{|y|^{\gamma +d -1}}dy = C rM_\eta(r).
\end{aligned}
\end{equation}
When $\varrho_r(t_0,x_0)\leq c_1$, then we simply have $J_1 \geq 0$ because of (\ref{nov2620}). For a bound on $J_{2}$, note that $|x_0| <2r(1+C_v c') < 3r$, and we use the inequality
\[
|y| \le |x_0+y| +|x_0| \le |x_0+y| + 3r  \le 4|x_0+y|
\]
for $|x_0+y|\geq r$. Assumption (\ref{aug608}) gives 
\[
\eta(y) \geq \eta(4|x_0+y|) \geq C \eta(|x_0+y|),
\]
thus  
\begin{equation}\label{nov1590}
\begin{aligned}
J_{2} &\geq C\int_{x_0+y\in G} \eta(x_0+y) dy = C\int_{y\in G} \eta(y) dy = C\int_{y\in G} M_\eta(|y|) \mu(dy) \\
&\geq CM_\eta(\varphi(r)) \mu(G) = \frac{C}{a} M_\eta(r) \mu(t)
\geq \sigma M_\eta(r)\mu(t),
\end{aligned}
\end{equation}
by choosing $\sigma$ small so that $\sigma \leq C/a$. The above estimates for $J_1$ and $J_2$ lead to (\ref{nov1564}) and~(\ref{nov1843})
in their respective cases.

It is straightforward to extend Proposition~\ref{Mholder} to equations with a forcing in the following way.
 \begin{lemma}\label{lem-jan202}
Under the above assumptions, solutions of 
\begin{align}\label{aug1042}
\d_t q + v(t,x)\cdot\nabla q + \mathscr{L} q = f(t,x),
\end{align}
with a uniformly bounded function $f(t,x)$ satisfy 
\begin{align}\label{aug1044}
|q(t,x) - q(t,y)|\le C t^{-\beta} [M_\eta(|x-y|)]^{-\beta}\Big[\sup_{0\le s\le 1,x\in\R^d}|q(s,x)|+t\sup_{0\le s\le 1,x\in\R^d}|f(s,x)|\Big].
\end{align}
\end{lemma}
\begin{proof} This is a simple consequence of the Duhamel formula. Let $\mathcal{K}(t,x)$ be 
the Green's function for the operator $v\cdot\nabla + \mathscr{L} $, so that the
the solution $q(t,x)$ to \eqref{homogmain} with the initial condition $q(0,x)=q_0(x)$ is
\begin{align*}
q(t,x) = \int_{\mathbb{R}^d} \mathcal{K}(t,x-y)q_0(y)dy.
\end{align*}
Then \eqref{Mholder} implies that for $0<t\leq 1$ and $|x_1-x_2|\le r_0$, we have
\begin{align*}
\frac{|q(t,x_1) - q(t,x_2)|t^{\beta}}{[M(|x_1-x_2|)]^{-\beta}} & = \biggl|\int_{\mathbb{R}^d}
\frac{(\mathcal{K}(t,x_1-y) -\mathcal{K}(t,x_2-y))t^{\beta} }{[M(|x_1-x_2|)]^{-\beta}}q_0(y)dy \biggr|\le C\|q_0\|_\infty.
\end{align*}
It follows that the kernel
\[
\tilde{K}(x_1,x_2,y,t) : =\frac{|\mathcal{K}(t,x_1-y) -\mathcal{K}(t,x_2-y)|t^{\beta} }{[M(|x_1-x_2|)]^{-\beta}}
\]
satisfies
\begin{equation}\label{aug1048}
\sup_{|x_1-x_2|\le r_0, 0<t\le 1}\int_{\R^d}\tilde{K}(x_1,x_2,y,t)dy\le C.
\end{equation}
Let now $q(t,x)$ be solution to (\ref{aug1042})  with $q(0,x)=0$. It is given by
the Duhamel formula
\begin{align*}
q(t,x) = 
\int_0^{t} \int_{\mathbb{R}^d} \mathcal{K}(t-s,x-y) f(s,y)dy ds,
\end{align*}
and we can write, using (\ref{aug1048}):
\begin{align*}
\frac{|q(t,x_1) - q(t,x_2)|}{[M(|x_1-x_2|)]^{-\beta}} & \leq  
\int_0^{t} \int_{\mathbb{R}^d}\frac{| \mathcal{K}(t-s,x_1-y) - \mathcal{K}(t-s,x_2-y) | }{[M(|x_1-x_2|)]^{-\beta}} f(s,y)dy ds\\
&\leq 
\|f\|_{L^\infty} \int_0^t \frac{1}{(t-s)^{\beta}} 
\int_{\mathbb{R}^d} \tilde{K}(x_1,x_2,y,t-s) dyds
 \le Ct^{1-\beta}||f||_{L^\infty},
\end{align*}
and (\ref{aug1044}) follows. 
\end{proof}

\subsection{The end of the proof of Theorem~\ref{thm-holder}}

In our case, $\rho(t,x)$ satisfies (\ref{main}), which is of the form (\ref{aug1042}) 
with $k(x,z,t) = \psi(z)\rho(t,x)$, with a uniformly bounded forcing term $G\rho$ in the right side,
due to Proposition~\ref{lem:aug1-02} and (\ref{aug130}).   
As $\rho(t,x)$ obeys the uniform upper and lower bounds in Proposition~\ref{lem:aug1-02},  the bounds~(\ref{aug604}) 
on the kernel~$k(x,z,t)$ hold
with $\eta(z)=\psi(z)$. Assumptions (\ref{aug606}), (\ref{jan206}) are then 
simply (\ref{aug108}), (\ref{aug1006}) respectively, while~(\ref{aug608}) holds due to the monotonicity of 
$rM(r)$ for $0\leq x\leq r_0$, see (\ref{aug1008}) and the comment following it. 

To see that the drift $u(t,x)$ in (\ref{main}) 
satisfies~(\ref{uest0}) and~\eqref{uest}, we first recall the decomposition~(\ref{udecomp}) that allows us to write
\begin{equation}
u=u_1+u_2 + I_0(t),~~\hbox{with }\partial_x u_1=G - \nu,~~\partial_xu_2={\cal L}\rho.
\end{equation} 
The uniform bound (\ref{aug130}) on $G$ implies
that $u_1$ obeys both (\ref{uest0}) and (\ref{uest}).   As for $u_2$, it can be written as
\begin{equation}\label{aug618}
u_2(t,x)={\cal L}\Phi=\int\psi(z)[\Phi(t,x)-\Phi(t,x+z)]dz,
\end{equation}
where $\Phi(t,x)$ is the mean-zero primitive of $\rho(t,x)$, as in (\ref{nov2916}).
Since $\Phi$ is Lipschitz because of the uniform bound on $\rho$, the $L^\infty$-bound
on $u_2$ follows from (\ref{aug1002}), 
and (\ref{uest0}) holds.  To verify~(\ref{uest}), we note that for any
$r>0$ we can write
\begin{equation}\label{aug620}
\begin{aligned}
u_2(t,x)-u_2(t,y)&=\int\psi(z)[\Phi(t,x)-\Phi(t,y)-
\Phi(t,x+z)+\Phi(t,y+z)]dz \\
&=\int_{|z|\le r}+\int_{|z|\ge r}=I_1(r)+I_2(r).
\end{aligned}
\end{equation}
These terms can be bounded as
\begin{equation}\label{aug622}
|I_2(r)|\le C|x-y|\int_{|z|\ge r}\psi(z)dz=C|x-y|M(r).
\end{equation}
As for $I_1$, we use assumption (\ref{aug1006}), that implies
\[
r\psi(r)\le \gamma_0M(r),~~\gamma_0:=\farc{r_0\psi(r_0)}{M(r_0)},~~\hbox{ for all $0<r<r_0$},
\]
so that, as long as $r\in(0,r_0)$, we have, taking $\gamma\in(0,1/2]$ as in (\ref{aug1008}), so that
$z^\gamma M(z)$ is an increasing function for $z\in(0,r_0)$:
\begin{equation}\label{aug624}
\begin{aligned}
|I_1(r)| &\le C\int_{|z|\le r}|z|\psi(z)dz \leq  C\gamma_0 \int_{|z|\le r} M(z) dz
\le C\gamma_0 \int_{|z|\le r} \farc{z^{\gamma}M(z)}{|z|^{\gamma}} dz \\
&\leq C\gamma_0r^{\gamma} M(r)  \int_{|z|\le r} \farc{1}{|z|^{\gamma}} dz = CrM(r).
\end{aligned}
\end{equation}
Setting $r =|x-y|$ implies that (\ref{uest}) holds.

The forcing term $G\rho$ in \eqref{main}  is uniformly bounded 
since
\[
\|G\rho\|_{L^\infty} \le C\|\rho\|_{L^\infty}^2,
\]
by (\ref{aug130}), thus Lemma~\ref{lem-jan202} finishes the proof.~$\Box$

\section{Global existence of smooth solutions}

In this section, we prove Theorem~\ref{thm-main}. 
We follow the strategy of \cite{shvydkoy2017eulerian3}  to show a uniform bound on $\rho_x$.
\begin{proposition}\label{prop-aug1304}
For each $T>0$ there exists $C_T>0$ so that for all~$0\le t\le T$
we have~$\|\rho_x(t,\cdot)\|_{L^\infty}\le C_T$.
\end{proposition}
The global existence of smooth solutions in Theroem~\ref{thm-main} will then follow by a bootstrap argument.
We begin the proof of Proposition~\ref{prop-aug1304}
with a nonlinear maximum principle for the operator~$\mathcal{L}$.  
\begin{lemma}\label{lemlowL}
Let $f \in C_b^1(\mathbb{R})$,  and $x_0$ be the maximal point where $|f'(x_0)| = \max_x|f'(x)|$,
then we have a lower bound  
\begin{equation}
f' (x_0)\mathcal{L}f' (x_0) \geq \frac{1}{2}Df'(x_0),
\end{equation}
where 
\begin{equation}
Df' (x) := \int_{\mathbb{R}} |f'(x)-f'(x+z)|^2 \psi(z)dz.
\end{equation}
\end{lemma}
\begin{proof} We have
\begin{equation*}
\begin{aligned}
f' (x_0)\mathcal{L}f' (x_0)  &= \int_{\mathbb{R}}(f'(x_0)^2 - f'(x_0)f'(x_0+z))\psi(z)dz\\
& = \frac{1}{2}\int_{\mathbb{R}} (f'(x_0)^2 - f'(x_0+z)^2)\psi(z)dz 
+ \frac{1}{2}\int_{\mathbb{R}} (f'(x_0) -f'(x_0+z))^2\psi(z) dz \\
&\geq  \frac{1}{2}Df'(x_0).
\end{aligned}
\end{equation*}
\end{proof}
The next step is a lower bound for $Df'$.
\begin{lemma}\label{lemlowD}
There exists $C>0$ that depends only on the kernel $\psi$
so that for all  $f \in C_b^1(\mathbb{R})$, we have a pointwise lower bound
\begin{equation}\label{aug1332}
Df' (x) \geq   |f'(x)|^2M\Big(\farc{C\|f\|_{L^\infty}}{|f'(x)|}\Big),~~\hbox{ for all $x\in\R$.}
\end{equation}
\end{lemma}
\begin{proof} As in the proof of Lemma~\ref{lem:aug14},  let $\chi(x)$ be a radially non-decreasing 
smooth cut-off function  
such that $\chi(x) = 0$ for $|x|\leq 1/2$ and $\chi(x) =1$ for $|x|\geq 1$. We write, using (\ref{aug1330}):
\begin{align*}
Df'(x)& \geq \int_{\mathbb{R}} |f'(x)-f'(x+z)|^2 \psi(z)\chi(z/R) dz\\
& \geq |f'(x)|^2 \int_{\mathbb{R}}\psi(z)\chi(z/R) dz - 2f'(x) \int_{\mathbb{R}} f'(x+z)\psi(z)\chi(z/R) dz\\
& \geq   |f'(x)|^2 \int_{|z|\geq R}\psi(z)dz +2f'(x) \int_{|z|\geq R/2} f(x+z) \d_z\bigl(\psi(z)\chi(z/R) \bigr)dz \\
&\geq 2|f'(x)|^2 M(R) - C|f'(x)| \frac{ \|f\|_{L^\infty}}{R} \int_{|z|\geq R/2}\psi(z)dz\\
& \geq 2|f'(x)|^2 M(R) - C|f'(x)| \|f\|_{L^\infty} \frac{M(R/2)}{R}.
\end{align*}
The doubling condition (\ref{aug1010}) for $M$ gives
\begin{align*}
Df'(x) \geq 2|f'(x)|^2 M(R) - C|f'(x)| \|f\|_{L^\infty} \frac{M(R)}{R}.
\end{align*}
Setting 
\[
R= \frac{C \|f\|_{L^\infty}}{|f'(x)|},
\]
we conclude that 
\begin{align*}
Df'(x) \geq |f'(x)|^2 M\Big(\farc{C\|f\|_{L^\infty}}{|f'(x)|}\Big),
\end{align*}
which is (\ref{aug1332}). 
\end{proof}


\subsubsection*{The proof of Proposition~\ref{prop-aug1304}}

We first note that there exists a time $\tau_0$ that depends on the $\|\rho_0\|_{Lip}$
and $\|u_0\|_{Lip}$ such that for all $0\le t\le \tau_0$ and all $x,y\in\T$ we have
\begin{equation}\label{aug1920}
|\rho(t,x)-\rho(t,y)|\le \farc{|x-y|}{2}\|\rho_0\|_{Lip}.
\end{equation}
Putting this estimate together with Theorem~\ref{thm-holder}, we conclude that
there exists a constant~$C>0$ that depends on $\|\rho_0\|_\infty$ and $\tau_0$ so that
we have
\begin{align}\label{aug1922}
|\rho(t,x) - \rho(t,y)|\le C [M(|x-y|)]^{-\beta},~~\hbox{for all $0\le t\le T$ and 
$x,y\in \T$.}
\end{align}
We take the derivative of  equation \eqref{main} and use (\ref{fullSys3}): 
\begin{equation}\label{aug1702}
\d_t \rho_x + (G+\mathcal{L}\rho)\rho_x + u \rho_{xx} + \rho_x G+\rho G_x = - \rho_x \mathcal{L}\rho - \rho \mathcal{L}\rho_x.
\end{equation}
Multiplying (\ref{aug1702}) by $\rho_x$ and evaluating at the maximal point $x_+$ of $\rho_x$, so that $\rho_{xx}(x_+) = 0$, we obtain
\begin{align}\label{aug1704}
\frac{1}{2}\d_t |\rho_x(x_+)|^2 =& -(G_x(x_+)\rho(x_+)\rho_x(x_+)+2G(x_+)|\rho_x(x_+)|^2) \\
\nonumber &- 2|\rho_x(x_+)|^2\mathcal{L}\rho(x_+) - \rho(x_+) \rho_x(x_+)\mathcal{L}\rho_x(x_+) = I+II+III.
\end{align}
By the uniform estimates (\ref{aug130}) on $G$ and (\ref{aug132}) on $G_x$, we
can bound the first term in the right side as
\begin{align*}
I\leq |G_x(x_+)\rho(x_+)\rho_x(x_+)+2G(x_+)|\rho_x(x_+)|^2| \leq C(1+ |\rho_x(x_+)|^2).
\end{align*}
Lemma \ref{lemlowL} together with a uniform lower bound on $\rho(t,x)$ in Proposition~\ref{lem:aug1-02} 
gives a bound for the dissipative term $III$
in the right side of (\ref{aug1704}):
\begin{align}\label{aug1706}
III \leq -\frac{1}{2}c_0 D\rho_x(x_+) .
\end{align}
To estimate there term $II$ in (\ref{aug1704}), we need to bound $\mathcal{L}\rho(x_+)$. 
We introduce a smooth symmetric cut-off function $\phi(x)$ such that $\phi(x) = 1$ for $|x|\leq 1$ and $\phi(x) = 0$ for $|x|\geq 2$, and write,
for any $x$ and $0<r<1/2$:
\begin{equation}\label{aug1708}
\begin{aligned} 
\mathcal{L}\rho (x) &= \int_{\mathbb{R}} \phi(\farc{z}{r}) (\rho(x)-\rho(x+z))\psi(z) dz +  
\int_{\mathbb{R}}(1- \phi(\farc{z}r)) (\rho(x)-\rho(x+z))\psi(z) dz\\
&=II_1+II_2.
\end{aligned}
\end{equation}
The second term above can be estimated using the $M$-H\"{o}lder estimate 
(\ref{aug1922}) for $\rho$: 
\begin{equation}
\begin{aligned}
&II_2\leq   
C\int_{|z|\geq r} M(|z|)^{-\beta}\psi(z)dz\le C{ M(r)^{1-\beta}}. 
\end{aligned}
\end{equation}
The estimate for  $II_1$ is more subtle.
Let  $\tilde{M}$, an odd extension of $M$, be the primitive of the even function~$(-\psi(z))$, then integration by parts gives
\begin{equation}\label{aug1712}
\begin{aligned}
II_1 &= \int_{|z|\leq 2r} \phi(\farc{z}r) (\rho(x)-\rho(x+z))(-\d_z\tilde{M}(z))dz\\
&\leq \frac{C}{r}\int_{r\leq |z|\leq 2r}  |\rho(x)-\rho(x+z)|\tilde{M}(z) dz - \int_{|z|\leq 2r}  \phi(\farc{z}r) \rho_z(x+z) \tilde{M}(z)dz
= II_{11} + II_{12}.
\end{aligned}
\end{equation}
Again by the $M$-H\"{o}lder estimate for $\rho$, we have, since $M(r)$ is decreasing
\begin{equation}\label{aug1714}
II_{11}\leq \frac{C}{r}
\int_{r\leq |z|\leq 2r}  |\rho(x)-\rho(x+z)|M(|z|) dz 
\leq \frac{C}{r}
\int_{r\leq |z|\leq 2r} M(|z|)^{1-\beta} dz \leq 
{C M(r)^{1-\beta}}. 
\end{equation}
Moreover, since $\psi(z)$ is even and $\tilde{M}$ is odd, we can write
\begin{equation}\label{aug1716}
\begin{aligned}
II_{12} &= \int_{|z|\leq 2r}  \phi(\farc{z}r)(\rho_x(x)-\rho_x(x+z)) \tilde{M}(z)dz\\
 &\leq C\int_{|z|\leq 2r} |\rho_x(x)-\rho_x(x+z)|\sqrt{\psi(z)} M(|z|)/\sqrt{\psi(z) } dz\\
 &\leq C\biggl( \int_{\mathbb{R}} |\rho_x(x)-\rho_x(x+z)|^2\psi(z)dz\biggr)^{1/2}\biggl(  \int_{|z|\leq 2r}\frac{M(|z|)^2 }{\psi(z)}dz\biggr)^{1/2}\\
 &=C \sqrt{D\rho_x(x)}\biggl(  \int_{|z|\leq 2r}\frac{M(|z|)^2 }{\psi(z)}dz\biggr)^{1/2}. 
\end{aligned}
\end{equation}
As a consequence of the lower bound on $\psi(r)$ in (\ref{aug1002}), and assumption (\ref{aug1008}) which implies
that~$x^{1/2} M(x)$ is non-decreasing, 
the integral in the right side of (\ref{aug1716}) can be bounded as
\begin{align*}
\int_{|z|\le 2r}  \frac{M(|z|)^2 }{\psi(z)}dz& \leq C \int_{|z|\le 2r} M(|z|)^2 |z|^{1-\alpha/2}dz \\
&\leq C r M(2r)^2  \int_{|z|\le 2r} |z|^{-\alpha/2} dz = C r^{2-\alpha/2} M(2r)^2 \leq C r^{2-\alpha}, 
\end{align*} 
as it follows from (\ref{aug1004}) that
\[
M(2r) \leq \farc{C}{(2r)^{\alpha/4}}.
\]
We conclude that
\begin{equation}\label{aug1720}
II_{12}\le C \sqrt{D\rho_x(x)} r^{1-\alpha/2}.
\end{equation}
Going back to $II$, we have shown that
\begin{equation}\label{aug1722}
\begin{aligned}
II &\leq C 
|\rho_x(x_+)|^2 M(r)^{1-\beta} + 
C|\rho_x(x_+)|^2 \sqrt{D\rho_x(x_+)} r^{1-\alpha/2} \\
&\leq  
C|\rho_x(x_+)|^2 M(r)^{1-\beta} + 
\frac{1}{4} c_0 D\rho_x(x_+) + C r^{2-\alpha} |\rho_x(x_+)|^4.
\end{aligned}
\end{equation}
Collecting our bounds for the three terms in the right side of (\ref{aug1704}), 
and using Lemma \ref{lemlowD}, we obtain
\begin{equation}\label{aug1902}
\begin{aligned} 
&\frac{1}{2}\d_t |\rho_x(x_+)|^2 
\le C(1+ |\rho_x(x_+)|^2)+
C|\rho_x(x_+)|^2 M(r)^{1-\beta} 
+ C r^{2-\alpha} |\rho_x(x_+)|^4- \frac{1}{4} c_0 D\rho_x(x_+)\\
&\le C(1+ |\rho_x(x_+)|^2)+
C|\rho_x(x_+)|^2 M(r)^{1-\beta} 
+ C r^{2-\alpha} |\rho_x(x_+)|^4- \frac{1}{4} c_0    
|\rho_x(x_+)|^2M\Big(\farc{C\|\rho\|_{L^\infty}}{|\rho_x(x_+)|}\Big).
\end{aligned}
\end{equation}
Let us choose 
\begin{equation}\label{aug1904}
r = \min(\farc{1}{2},r_c^{2/(2-\alpha)}),~~r_c:=
\frac{C \|\rho\|_{L^\infty}}{|\rho_x(x_+)| }.
\end{equation}
It follows from 
Lemma \ref{claimpow} that
\begin{align}\label{condr}
M(r)\leq C [1+M(r_c)]^{2/(2-\alpha)}.
\end{align}
Using this estimate in (\ref{aug1902}), as well as the definition (\ref{aug1904}) of $r$,
we obtain
\begin{align}\label{aug1908}
\frac{1}{2}\d_t |\rho_x(x_+)|^2 \leq & C(1+|\rho_x(x_+)|^2) + 
C|\rho_x(x_+)|^2 [1+M(r_c)]
^{2(1-\beta)/(2-\alpha)}
- C |\rho_x(x_+)|^2 M(r_c). 
\end{align}
Let us set $z(t)=|\rho_x(t,x_+(t))|^2$, then (\ref{aug1908}) is
\begin{align}\label{aug1910}
\farc{dz}{dt} \leq & C(1+z) +
Cz\Big[M\Big(\farc{C}{\sqrt{z}}\Big)\Big]^{2(1-\beta)/(2-\alpha)}
- C zM\Big(\farc{C}{\sqrt{z}}\Big). 
\end{align}
We choose $\alpha <2\beta$ so that $2(1-\beta)/(2-\alpha) < 1$, 
and then apply the Young's inequality, to obtain 
\begin{align}\label{aug1911}
\farc{dz}{dt} \leq & C(1+z) 
- \farc{C z}{2}M\Big(\farc{C}{\sqrt{z}}\Big). 
\end{align}
As $M(z)\to+\infty$ as $z\to 0$, the maximum principle implies that $z(t)$ remains finite
at all times, and the proof of Proposition~\ref{prop-aug1304} is complete.

\subsubsection*{The proof of Theorem \ref{thm-main} by bootstrapping }

Proposition~\ref{prop-aug1304} tells that $|\d_x \rho(t,x)|$ stays bounded, hence so does $|\d_x G(t,x)|$ by (\ref{dec510}). Therefore, 
$|\d_x u (t,x)|$ stays bounded because
\begin{equation}\label{dec523}
\begin{aligned}
|\d_x u (t,x)| &\le | G(t,x)| + |\mathcal{L}\rho(t,x)|\\
&\le C + \int_{|z|\le 1} |\rho(t,x)-\rho(t,x+z)| \psi(z) dz + \int_{|z|\ge 1} |\rho(t,x)-\rho(t,x+z)| \psi(z) dz\\
&\le C + \|\d_x \rho(t,\cdot)\|_{L^{\infty}} \int_{|z|\le 1}|z| \psi(z) dz  + 2\|\rho(t,\cdot)\|_{L^{\infty}} M(1)\le C,
\end{aligned}
\end{equation}
with $|z|\psi(z)$ being locally integrable by (\ref{aug1002}). Now we differentiate (\ref{main}) in $x$ and rearrange it to be
\begin{equation}\label{bt1}
\d_t(\rho_x)+ u\d_x( \rho_x)+ \rho \mathcal{L}( \rho_x) = -G_x \rho - G\rho_x - \rho_x\mathcal{L}\rho - u_x \rho_x.
\end{equation}
Note that the right side of (\ref{bt1}) is a bounded forcing term, thus (\ref{bt1}), viewed as an equation for $\rho_x$, 
lies in the class of linear integro-differential equations (\ref{fullholder}) and $\rho_x$ satisfies the $M$-H\"{o}lder estimate (\ref{aug1922}). Now let us repeat the proof of Proposition \ref{prop-aug1304}. We take the derivative of \eqref{aug1702} and use (\ref{fullSys3}) to get
\begin{equation}\label{dec600}
\begin{aligned}
\d_t \rho_{xx} +u\rho_{xxx}=& -\big(G_{xx}\rho + 3G\rho_{xx} + 3\rho_{xx}\mathcal{L}\rho +3G_x\rho_x\big)\\
&-3\rho_x\mathcal{L}\rho_x - \rho \mathcal{L}\rho_{xx}.
\end{aligned}
\end{equation}
Multiplying (\ref{dec600}) by $\rho_{xx}$, at the maximal point $x_+$ of $\rho_{xx}$, so that $\rho_{xxx}(x_+) = 0$, we obtain
\begin{equation}\label{dec601}
\begin{aligned}
\frac{1}{2}\d_t   |\rho_{xx}(x_+)|^2 
=& -\big(G_{xx}(x_+)\rho(x_+)\rho_{xx}(x_+)+(3G(x_+) + 3\mathcal{L}\rho(x_+)) |\rho_{xx}(x_+)|^2  \\
&+	3G_x(x_+)\rho_x(x_+)\rho_{xx}(x_+) \big) \\
&- 3\rho_{xx}(x_+)\rho_x(x_+)\mathcal{L}\rho_x(x_+) - \rho(x_+) \rho_{xx}(x_+)\mathcal{L}\rho_{xx}(x_+) = I'+II'+III'.
\end{aligned}
\end{equation}
We see that (\ref{dec601}) has the same structure as  (\ref{aug1704}). The estimate (\ref{dec510}), and the boundedness of $u_x$ and $\mathcal{L}\rho$ give that
\begin{equation}
I' \le C (1+|\rho_{xx}(x_+)|^2).
\end{equation}
Again Lemma \ref{lemlowL} and a uniform lower bound for $\rho(t,x)$ gives a bound
\begin{align}
III' \le -\frac{1}{2}c_0 D\rho_{xx}(x_+) .
\end{align}
And for $II'$, follow the proof of Proposition \ref{prop-aug1304} and note that $\rho_x(t,x)$ is bounded, to get
\begin{equation}
\begin{aligned}
II' &\le C |\rho_{xx}(x_+)| M(r)^{1-\beta} + 
C|\rho_{xx}(x_+)| \sqrt{D\rho_{xx}(x_+)} r^{1-\alpha/2} \\
&\leq  
C|\rho_{xx}(x_+)| M(r)^{1-\beta} + 
\frac{1}{4} c_0 D\rho_{xx}(x_+) + C r^{2-\alpha} |\rho_{xx}(x_+)|^2.
\end{aligned}
\end{equation}
We can now obtain an inequality similar to (\ref{aug1911}), which implies the that $|\d_x^2 \rho(t,x)|$ stays bounded, and thus $|\d_x^2 u(t,x)|$ stays bounded. The arguments above can be iterated due to two reasons: first, the analog of (\ref{bt1}) for the higher order derivatives, is always in the form of an integro-differential equation 
\begin{align}\label{bt2}
\d_t (\rho^{(n)}) +  u \d_x (\rho^{(n)}) + \rho\mathcal{L}(\rho^{(n)}) = f_n(G^{(k)}, \rho^{(k)}, u^{(k)}, \mathcal{L}
\rho^{(k-1)}), ~~n\in \N, 1\le k\le n,
\end{align}
where $f_n$ is a polynomial function depending on $G^{(k)}, \rho^{(k)}, u^{(k)}, \mathcal{L}
\rho^{(k-1)}, 1\le k\le n$, and thus~$f_n$ is a bounded forcing term. 
Second, when we take $(n+1)$-th derivative of (\ref{main}) and use (\ref{fullSys3}), it can always be put into the form
\begin{equation}\label{bt3}
\begin{aligned}
\d_t (\rho^{(n+1)}) + u\rho^{(n+2)} =& -g(G^{(k)}, \rho^{(k)},\mathcal{L}\rho^{(k-2)} )-C\rho_x\mathcal{L}\rho^{(n)}-\rho\mathcal{L}\rho^{(n+1)}, ~~n\in \N, 2\le k\le n+1,
\end{aligned}
\end{equation}
and the first term $|g|\le C(1+|\rho^{(n+1)}(t,x)|)$. In (\ref{bt3}), we look at the maximal 
point of~$|\rho^{(n+1)}|$ and so on, to get that $|\rho^{(n+1)}(t,x)|$ stays bounded, so is $|u^{(n+1)}(t,x)|$.
Estimates like in (\ref{aug1911}) also imply the local existence of $\rho^{(n)}, u^{(n)}$. By bootstrapping, we conclude that $\rho(t,x),u(t,x)$ stays smooth for all times.~$\qed$

\end{document}